\documentclass{amsart}

\usepackage{hyperref, amssymb}
\usepackage{color}
\usepackage[all]{xy}

\theoremstyle{plain}
\newtheorem{theorem}{Theorem}[section]
\newtheorem{proposition}[theorem]{Proposition}
\newtheorem{lemma}[theorem]{Lemma}
\newtheorem{corollary}[theorem]{Corollary}

\theoremstyle{definition}

\newtheorem{remark}[theorem]{Remark}

\newcommand{\f}{\varphi}

\newcommand{\CC}{\mathbb C}
\newcommand{\PP}{\mathbb P}

\newcommand{\MM}{\mathbf M}

\newcommand{\E}{{\mathcal E}}
\newcommand{\F}{{\mathcal F}}

\newcommand{\I}{{\mathcal I}}

\def\O{\mathcal O}

\newcommand{\Ker}{{\mathcal Ker}}
\newcommand{\Coker}{{\mathcal Coker}}
\newcommand{\Image}{{\mathcal Im}}

\newcommand{\Ext}{\operatorname{Ext}}
\newcommand{\Hilb}{\operatorname{Hilb}}
\newcommand{\Hom}{\operatorname{Hom}}
\def\H{\operatorname{H}}
\newcommand{\M}{\operatorname{M}}

\newcommand{\Poly}{\operatorname{P}}
\newcommand{\p}{\operatorname{p}}

\newcommand{\tensor}{\otimes}

\newcommand{\lra}{\longrightarrow}

\begin{document}

\title[Moduli of sheaves supported on curves of genus $4$ in $\PP^1 \times \PP^1$]
{On the geometry of the moduli space of sheaves supported on curves of genus four contained in a quadric surface}

\author{Mario Maican}
\address{Institute of Mathematics of the Romanian Academy, Calea Grivitei 21, Bucharest 010702, Romania}

\email{maican@imar.ro}

\begin{abstract}
We study the moduli space of stable sheaves of Euler characteristic $1$
supported on curves of bidegree $(3, 3)$ contained in a smooth quadric surface.
We show that this moduli space is rational.
We compute its Betti numbers by studying the variation of the moduli spaces of $\alpha$-semi-stable pairs.
\end{abstract}


\maketitle

\section{Introduction}
\label{introduction}

Given a polynomial $P(m, n)$ in two variables with rational coefficients we denote by $\M(P)$ the moduli space
of sheaves, on the product of two complex projective lines $\PP^1 \times \PP^1$,
that have Hilbert polynomial $P$ and are semi-stable with respect to the polarization $\O(1, 1)$.
This paper will be focused on the geometry of $\MM = \M(3m + 3n + 1)$.
A sheaf $\F \in \MM$ satisfies $\chi(\F) = 1$ and is supported on a curve of bidegree $(3, 3)$.
The semi-stability of $\F$ is given by two conditions: $\F$ is pure, meaning that there are no subsheaves with support of dimension zero,
and for any subsheaf $\E \subset \F$ with Hilbert polynomial $P_{\E}(m, n) = rm + sn + t$ we have the inequality of slopes
\[
\p(\E) = \frac{t}{r + s} \le \frac{1}{6} = \p(\F).
\]
From the general construction \cite{simpson} of moduli spaces of semi-stable sheaves with a fixed Hilbert polynomial we know that $\MM$ is projective
and from \cite{lepotier} we know that $\MM$ is smooth, irreducible, of dimension $19$.
We refer to the introductory section of \cite{genus_two} for more details about moduli spaces of sheaves on $\PP^1 \times \PP^1$ with linear Hilbert polynomial.

For a projective variety $X$ we define the Poincar\'e polynomial
\[
\Poly(X)(\xi) = \sum_{i \ge 0} \dim_{\mathbb Q}^{} \H^i(X, {\mathbb Q}) \xi^{i/2}.
\]
The varieties with which we will be concerned in this paper will have no odd homology, so the above will truly be a polynomial expression.

\begin{theorem}
\label{poincare_polynomial}
The Poincar\'e polynomial of $\MM$ is
\begin{multline*}
\xi^{19} + 3\xi^{18} + 10\xi^{17} + 22\xi^{16} + 41\xi^{15} + 53\xi^{14} + 60\xi^{13} + 62\xi^{12} + 63\xi^{11} + 63\xi^{10} \\
+ 63\xi^9 + 63\xi^8 + 62\xi^7 + 60\xi^6 + 53\xi^5 + 41\xi^4 + 22\xi^3 + 10\xi^2 + 3\xi + 1.
\end{multline*}
The Euler characteristic of $\MM$ is $756$.
\end{theorem}

\noindent
The Poincar\'e polynomial of $\M(3m + 2n + 1)$ was computed in \cite{choi_katz_klemm}, \cite{genus_two} and \cite{chung_moon_2}
by different methods. The Poincar\'e polynomial of $\M(4m + 2n + 1)$ was computed in \cite{genus_three}.
To prove Theorem \ref{poincare_polynomial} we will use the same method we used in \cite{genus_two} and \cite{genus_three},
namely variation of moduli spaces of $\alpha$-semi-stable pairs. A \emph{pair} $(\Gamma, \F)$ consists of a coherent sheaf $\F$
on $\PP^1 \times \PP^1$ and a $1$-dimensional vector subspace $\Gamma \subset \H^0(\F)$. Fix a positive rational number $\alpha$.
We restrict our attention to the case when $\F$ has linear Hilbert polynomial $P_{\F}(m, n) = rm + sn + t$.
Then the \emph{$\alpha$-slope} of $(\Gamma, \F)$ is
\[
\p_{\alpha}(\Gamma, \F) = \frac{\alpha + t}{r + s}.
\]
The $\alpha$-semi-stability of $(\Gamma, \F)$ is given by three conditions: $\F$ is pure, for any subsheaf $\E \subset \F$ we have
$\p(\E) \le \p_{\alpha}(\F)$ and for any subpair $(\Gamma, \E) \subset (\Gamma, \F)$ we have $\p_{\alpha}(\Gamma, \E) \le \p_{\alpha}(\Gamma, \F)$.
Moduli spaces $\M^{\alpha}(P)$ of $\alpha$-semi-stable pairs with fixed Hilbert polynomial $P$ have been constructed in \cite{lepotier_asterisque} and \cite{he}.
The positive real axis can be divided into finitely many intervals $(\alpha_0, \alpha_1)$, $(\alpha_1, \alpha_2)$, $ \ldots$, $(\alpha_k, \alpha_{k + 1})$, where $\alpha_0 = 0$, $\alpha_{k + 1} = \infty$, $\alpha_i$ are rational numbers,
such that $\M^{\alpha}(P)$ remains unchanged when $\alpha$ varies in an interval, and $\M^{\alpha}(P)$ changes when $\alpha$ crosses the
boundaries between intervals (which are called \emph{walls}). We write $\M^{0+}(P) = \M^{\alpha}(P)$ for $\alpha \in (0, \alpha_1)$
and $\M^{\infty}(P) = \M^{\alpha}(P)$ for $\alpha \in (\alpha_k, \infty)$.

The Poincar\'e polynomial of $\MM$ is related to the Poincar\'e polynomials of $\M^{0+}(3m + 3n + 1)$ and $\M^{0+}(3m + 3n - 1)$
by the formula from Proposition \ref{poincare_formula}. These in turn can be related to the Poincar\'e polynomials of
$\M^{\infty}(3m + 3n + 1)$ and $\M^{\infty}(3m + 3n - 1)$, which are certain flag Hilbert schemes.
In Section \ref{variation_1} we explain the relationship by a series of blow-ups and blow-downs between $\M^{\infty}(3m + 3n + 1)$
and $\M^{0+}(3m + 3n + 1)$.
In Section \ref{variation_2} we explain the relationship between $\M^{\infty}(3m + 3n - 1)$ and $\M^{0+}(3m + 3n - 1)$.
In Section \ref{preliminaries} we prove that $\M^{\infty}(3m + 3n + 1)$ is smooth.


\section{Preliminaries}
\label{preliminaries}

The purpose of this section is to prove Proposition \ref{Hilbert_scheme}, which will be needed later in the analysis of $\M^{0+}(3m + 3n + 1)$.

For any coherent sheaf $\F$ on $\PP^1 \times \PP^1$, \cite[Lemma 1]{buchdahl} provides a spectral sequence converging to $\F$,
with first level $E_1$ given by the tableau
\begin{equation}
\label{E1}
\xymatrix
{
\H^2(\F(-1, -1)) \tensor \O(-1, -1) = E_1^{-2, 2} \ar[r] & E_1^{-1, 2} \ar[r] & E_1^{0, 2} = \H^2(\F) \tensor \O \\
\H^1(\F(-1, -1)) \tensor \O(-1, -1) = E_1^{-2, 1} \ar[r] & E_1^{-1, 1} \ar[r] & E_1^{0, 1} = \H^1(\F) \tensor \O \\
\H^0(\F(-1, -1)) \tensor \O(-1, -1) = E_1^{-2, 0} \ar[r] & E_1^{-1, 0} \ar[r] & E_1^{0, 0} = \H^0(\F) \tensor \O
}
\end{equation}
All the remaining $E_1^{ij}$ are zero. The sheaves $E_1^{-1, j}$ fit into the exact sequences
\begin{equation}
\label{Ej}
\H^j(\F(0, -1)) \tensor \O(0, -1) \lra E_1^{-1, j} \lra \H^j(\F(-1, 0)) \tensor \O(-1, 0).
\end{equation}

\begin{lemma}
\label{vanishing}
Let $Z \subset \PP^1 \times \PP^1$ be a zero-dimensional subscheme of length $4$. Then $\H^1(\I_Z(3, 3)) = 0$.
\end{lemma}

\begin{proof}
Assume first that $Z$ is contained in a line of bidegree $(1, 0)$. Then we have a resolution
\[
0 \lra \O(-1, -4) \lra \O(-1, 0) \oplus \O(0, -4) \lra \I_Z \lra 0.
\]
From the long exact sequence in cohomology associated to the short exact sequence
\[
0 \lra \O(2, -1) \lra \O(2, 3) \oplus \O(3, -1) \lra \I_Z(3, 3) \lra 0
\]
we get $\H^1(\I_Z(3, 3)) = 0$. The same argument applies if $Z$ is contained in a line of bidegree $(0, 1)$.
For the remaining part of this proof we will assume that $Z$ is not contained in a line.
We apply the spectral sequence (\ref{E1}) to the sheaf $\F = \I_Z(1, 1)$.
By hypothesis, $\H^0(\I_Z(1, 0)) = 0$ and $\H^0(\I_Z(0, 1)) = 0$ hence, from sequence (\ref{Ej}) with $j = 0$, we get $E_1^{-1, 0} = 0$.
Also, $\H^0(\I_Z) = 0$, hence $E_1^{-2, 0} = 0$.
By Serre duality
\begin{align*}
\H^2(\I_Z(1, 1)) & \simeq \Hom(\I_Z(1, 1), \omega)^* = 0, \qquad \H^2(\I_Z) \phantom{(0, 1)} \simeq \Hom(\I_Z, \omega)^* = 0, \\
\H^2(\I_Z(1, 0)) & \simeq \Hom(\I_Z(1, 0), \omega)^* = 0, \qquad \H^2(\I_Z(0, 1)) \simeq \Hom(\I_Z(0, 1), \omega)^* = 0,
\end{align*}
hence $E_1^{0, 2} = 0$, $E_1^{-2, 2} = 0$ and, from sequence (\ref{Ej}) with $j = 2$ we get $E_1^{-1, 2} = 0$. We have
\begin{align*}
\dim_{\CC}^{} \H^1(\I_Z) \phantom{(0, 1)} & = - \chi(\I_Z) \phantom{(0, 1)} = 3, \\
\dim_{\CC}^{} \H^1(\I_Z(1, 0)) & = - \chi(\I_Z(1, 0)) = 2, \\
\dim_{\CC}^{} \H^1(\I_Z(0, 1)) & = - \chi(\I_Z(0, 1)) = 2.
\end{align*}
Denote $d = \dim_{\CC}^{} \H^1(\I_Z(1, 1)) = \dim_{\CC}^{} \H^0(\I_Z(1, 1))$.
Since $Z$ is not contained in a line, $Z$ is not contained in the intersection of two conics, hence $d \le 1$.
Tableau (\ref{E1}) now takes the simplified form
\[
\xymatrix
{
0 & 0 & 0 \\
3\O(-1, -1) \ar[r]^-{\f_1} & E_1^{-1, 1} \ar[r]^-{\f_2} & d\O \\
0 & 0 & d\O
}
\]
The convergence of the spectral sequence means that $\f_2$ is surjective and that we have an exact sequence
\begin{equation}
\label{convergence}
0 \lra \Ker(\f_1) \overset{\f_3}{\lra} d\O \lra \I_Z(1, 1) \lra \Ker(\f_2)/\Image(\f_1) \lra 0.
\end{equation}
From this we can compute the Hilbert polynomial of $E_1^{-1, 1}$:
\begin{align*}
P_{E_1^{-1, 1}} & = P_{\Ker(\f_2)} + P_{d\O} \\
& = P_{\Image(\f_1)} + P_{\Ker(\f_2)/\Image(\f_1)} + P_{d\O} \\
& = P_{3\O(-1, -1)} - P_{\Ker(\f_1)} + P_{\I_Z(1, 1)} - P_{d\O} + P_{\Ker(\f_1)} + P_{d\O} \\
& = P_{3\O(-1, -1)} + P_{\I_Z(1, 1)} \\
& = P_{3\O(-1, -1)} + P_{\O(1, 1)} - P_{\O_Z} \\
& = 3mn + (m + 2)(n + 2) - 4 \\
& = 4mn + 2m + 2n.
\end{align*}
On the other hand we have the exact sequence (\ref{Ej}) with $j = 1$:
\[
2\O(0, -1) \lra E_1^{-1, 1} \lra 2\O(-1, 0).
\]
Since $P_{E_1^{-1, 1}} = P_{2\O(0, -1)} + P_{2\O(-1, 0)}$, this sequence is also exact on the left and on the right, hence
$E_1^{-1, 1} \simeq 2\O(0, -1) \oplus 2\O(-1, 0)$.

Assume that $d = 0$. Then the exact sequence (\ref{convergence}) yields the resolution
\[
0 \lra 3\O(-1, -1) \lra 2\O(0, -1) \oplus 2\O(-1, 0) \lra \I_Z(1, 1) \lra 0.
\]
From the long exact cohomology sequence we get $\H^1(\I_Z(3, 3)) = 0$.

Assume that $d = 1$. Then $\Ker(\f_1) = 0$, otherwise $\Coker(\f_3)$ would be a non-zero torsion subsheaf of $\I_Z(1, 1)$,
which is absurd.
Let $V_1$ and $V_2$ be vector spaces over $\CC$ of dimension $2$ and make the identification $\PP^1 \times \PP^1 = \PP(V_1) \times \PP(V_2)$.
We choose a basis $\{ x, y \}$ of $V_1^*$ and a basis $\{ z, w \}$ of $V_2^*$.
Since $\f_2$ is surjective and $\f_1$ is injective, it is easy to see that $\f_2$ has one of the following canonical forms:
\[
\left[
\begin{array}{cccc}
1 \tensor z & 0 & x \tensor 1 & y \tensor 1
\end{array}
\right],
\]
\[
\left[
\begin{array}{cccc}
1 \tensor z & 1 \tensor w & x \tensor 1 & 0
\end{array}
\right],
\]
\[
\left[
\begin{array}{cccc}
1 \tensor z & 1 \tensor w & x \tensor 1 & y \tensor 1
\end{array}
\right].
\]
In each case it is clear that $\f_2(2, 2)$ is surjective on global sections.
From the long exact cohomology sequence associated to the short exact sequence
\[
0 \lra \Ker(\f_2)(2, 2) \lra 2\O(2, 1) \oplus 2\O(1, 2) \xrightarrow{\f_2(2, 2)} \O(2, 2) \lra 0
\]
we get $\H^1(\Ker(\f_2)(2, 2)) = 0$.
The exact sequence (\ref{convergence}) becomes
\[
0 \lra \O \lra \I_Z(1, 1) \lra \Ker(\f_2)/3\O(-1, -1) \lra 0.
\]
From long exact cohomology sequences we deduce that $\H^1(\I_Z(3, 3)) = 0$.
\end{proof}

\begin{proposition}
\label{Hilbert_scheme}
Let $\H$ be the flag Hilbert scheme of zero-dimensional subschemes of length $4$ contained in curves of bidegree $(3, 3)$ in $\PP^1 \times \PP^1$.
Then $\H$ is a $\PP^{11}$-bundle over $\Hilb_{\PP^1 \times \PP^1}(4)$, so it is smooth.
\end{proposition}

\begin{proof}
The fiber of the canonical morphism $\H \to \Hilb_{\PP^1 \times \PP^1}(4)$ over a zero-dimensional subscheme $Z$
is $\PP(\H^0(\I_Z(3, 3)))$. By virtue of Lemma \ref{vanishing}, the exact sequence
\[
0 \lra \I_Z(3, 3) \lra \O(3, 3) \lra \O_Z \lra 0
\]
is also exact on global sections. From this we get $\dim_{\CC}^{} \H^0(\I_Z(3, 3)) = 12$.
\end{proof}


\section{Variation of $\M^{\alpha}(3m + 3n + 1)$}
\label{variation_1}

\begin{proposition}
\label{walls}
With respect to the polynomial $P(m, n) = 3m + 3n + 1$ there are exactly three walls at $\alpha_1 = 2$, $\alpha_2 = 5$ and $\alpha_3 = 11$.
\end{proposition}

\begin{proof}
As at \cite[Proposition 5.2]{genus_two}, we need to solve the equation
\begin{equation}
\label{alpha}
\frac{\alpha + t}{r + s} = \frac{\alpha + 1}{6}
\end{equation}
with integers $0 \le r \le 3$, $0 \le s \le 3$, $t \ge r + s - rs$, and $\alpha > 0$ a rational number, the case when $(r, s) = (3, 3)$ being excluded.
Assume that $(r, s) = (3, 2)$ or $(2, 3)$ and $t \ge -1$.
Equation (\ref{alpha}) becomes $\alpha = 5 - 6t$, which has solutions $\alpha_3 = 11$ for $t = -1$ and $\alpha_2 = 5$ for $t = 0$.
Assume that $(r, s) = (2, 2)$, $t \ge 0$. Equation (\ref{alpha}) becomes $\alpha = 2 - 3t$, which has solution $\alpha_1 = 2$ for $t = 0$.
For all other choices of $r$ and $s$ equation (\ref{alpha}) has no positive solution in $\alpha$.
\end{proof}

\noindent
Denote $\MM^{\alpha} = \M^{\alpha}(3m + 3n + 1)$. For $\alpha \in (11, \infty)$ we write $\MM^{\alpha} = \MM^{\infty}$.
For $\alpha \in (5, 11)$ we write $\MM^{\alpha} = \MM^{5+} = \MM^{11-}$.
For $\alpha \in (2, 5)$ we write $\MM^{\alpha} = \MM^{2+} = \MM^{5-}$.
For $\alpha \in (0, 2)$ we write $\MM^{\alpha} = \MM^{0+}$.
The inclusions of sets of $\alpha$-semi-stable pairs induce the flipping diagrams
\[
\xymatrix
{
\MM^{\infty} \ar[dr]_-{\rho_{\infty}} & & \MM^{11-} \ar[dl]^-{\rho_{11}} \ar@{=}[r] & \MM^{5+} \ar[dr]_-{\rho_{5+}} & & \MM^{5-} \ar[dl]^-{\rho_{5-}} \ar@{=}[r] & \MM^{2+} \ar[dr]_-{\rho_2} & & \MM^{0+} \ar[dl]^-{\rho_0} \\
& \MM^{11} & & & \MM^5 & & & \MM^2
}
\]
in which all maps are birational.

The following proposition is a particular case of \cite[Proposition B8]{pandharipande_thomas}:

\begin{proposition}
\label{M^infinity}
The variety $\MM^{\infty}$ is isomorphic to the flag Hilbert scheme of zero-dimensional subschemes of length $4$ contained in curves of bidegree $(3, 3)$
in $\PP^1 \times \PP^1$.
\end{proposition}

\begin{corollary}
\label{rational}
The variety $\MM$ is rational.
\end{corollary}

\begin{proof}
By Propositions \ref{Hilbert_scheme} and \ref{M^infinity}, $\MM^{\infty}$ is rational.
Since $\MM^{0+}$ is birational to $\MM^{\infty}$, also $\MM^{0+}$ is rational.
The forgetful morphism $\MM^{0+} \to \MM$, $(\Gamma, \F) \mapsto [\F]$ is birational because, for a generic sheaf $\F \in \MM$,
we have $\H^0(\F) \simeq \CC$. Thus, $\MM$ is also rational.
\end{proof}

\begin{remark}
\label{flipping_base}
From the proof of Proposition \ref{walls} we see that the strictly $\alpha$-semi-stable locus in $\MM^{11}$ has two connected components:
\[
\M^{0+}(3m + 2n - 1) \times \M(n + 2) \qquad \text{and} \qquad \M^{0+}(2m + 3n - 1) \times \M(m + 2).
\]
According to \cite[Theorem 2.2]{genus_two}, a point in $\M^{0+}(3m + 2n - 1)$ is of the form $(\Gamma_1, \E_1)$,
where $\E_1$ is the structure sheaf of a curve of bidegree $(2, 3)$ and $\Gamma_1 = \H^0(\E_1)$. Thus, $\M^{0+}(3m + 2n - 1) \simeq \PP^{11}$.
A point in $\M(n + 2)$ is of the form $\O_L(0, 1)$, where $L \subset \PP^1 \times \PP^1$ is a line of bidegree $(1, 0)$. Thus, $\M(n + 2) \simeq \PP^1$.

The strictly $\alpha$-semi-stable locus in $\MM^5$ has two connected components:
\[
\M^{0+}(3m + 2n) \times \M(n + 1) \qquad \text{and} \qquad \M^{0+}(2m + 3n) \times \M(m + 1).
\]
According to \cite[Proposition 3.3]{genus_three}, the points in $\M^{0+}(3m + 2n)$ are of the form $(\Gamma_3, \E_3)$, where $\Gamma_3 = \H^0(\E_3)$
and $\E_3$ has resolution
\begin{equation}
\label{E_3}
0 \lra \O(-2, -2) \oplus \O(-1, -3) \overset{\f}{\lra} \O(-1, -2) \oplus \O \lra \E_3 \lra 0,
\end{equation}
with $\f_{11} \neq 0$, $\f_{12} \neq 0$. Moreover, $\M^{0+}(3m + 2n)$ is isomorphic to the universal quintic of bidegree $(2, 3)$,
so it is a $\PP^{10}$-bundle over $\PP^1 \times \PP^1$.
A point in $\M(n + 1)$ is of the form $\O_L$, hence $\M(n + 1) \simeq \PP^1$.

The strictly $\alpha$-semi-stable locus in $\MM^2$ is of the form $\M^{0+}(2m + 2n) \times \M(m + n + 1)$.
By \cite[Remark 5.4]{genus_two}, a point in $\M^{0+}(2m + 2n)$ is of the form $(\Gamma_5, \E_5)$,
where $\E_5$ is the structure sheaf of a curve of bidegree $(2, 2)$ and $\Gamma_5 = \H^0(\E_5)$.
Thus, $\M^{0+}(2m + 2n) \simeq \PP^8$. According to \cite[Proposition 11]{ballico_huh}, a sheaf $\E_6 \in \M(m + n + 1)$
is the structure sheaf of a curve of bidegree $(1, 1)$. Thus, $\M(m + n + 1) \simeq \PP^3$.
\end{remark}

\noindent
Consider the flipping loci
\begin{align*}
F^{\infty}_1 & = \rho_{\infty}^{-1}(\M^{0+}(3m + 2n - 1) \times \M(n + 2)) \subset \MM^{\infty}, \\
F^{\infty}_2 & = \rho_{\infty}^{-1}(\M^{0+}(2m + 3n - 1) \times \M(m + 2)) \subset \MM^{\infty}, \\
F^{\infty} & = F^{\infty}_1 \cup F^{\infty}_2, \\
F^{11}_1 & = \rho_{11}^{-1}(\M^{0+}(3m + 2n - 1) \times \M(n + 2)) \subset \MM^{11-}, \\
F^{11}_2 & = \rho_{11}^{-1}(\M^{0+}(2m + 3n - 1) \times \M(m + 2)) \subset \MM^{11-}, \\
F^{11} & = F^{11}_1 \cup F^{11}_2, \\
F^{5+}_1 & = \rho_{5+}^{-1}(\M^{0+}(3m + 2n) \times \M(n + 1)) \subset \MM^{5+}, \\
F^{5+}_2 & = \rho_{5+}^{-1}(\M^{0+}(2m + 3n) \times \M(m + 1)) \subset \MM^{5+}, \\
F^{5+} & = F^{5+}_1 \cup F^{5+}_2, \\
F^{5-}_1 & = \rho_{5-}^{-1}(\M^{0+}(3m + 2n) \times \M(n + 1)) \subset \MM^{5-}, \\
F^{5-}_2 & = \rho_{5-}^{-1}(\M^{0+}(2m + 3n) \times \M(m + 1)) \subset \MM^{5-}, \\
F^{5-} & = F^{5-}_1 \cup F^{5-}_2, \\
F^2 & = \rho_2^{-1}(\M^{0+}(2m + 2n) \times \M(m + n + 1)) \subset \MM^{2+}, \\
F^0 & = \rho_0^{-1}(\M^{0+}(2m + 2n) \times \M(m + n + 1)) \subset \MM^{0+}.
\end{align*}
Over a point $(\Lambda_1, \Lambda_2) \in \M^{0+}(3m + 2n - 1) \times \M(n + 2)$, $F^{\infty}_1$ has fiber $\PP(\Ext^1(\Lambda_1, \Lambda_2))$
and $F^{11}_1$ has fiber $\PP(\Ext^1(\Lambda_2, \Lambda_1))$.
Over a point $(\Lambda_1', \Lambda_2') \in \M^{0+}(2m + 3n - 1) \times \M(m + 2)$, $F^{\infty}_2$ has fiber $\PP(\Ext^1(\Lambda_1', \Lambda_2'))$
and $F^{11}_2$ has fiber $\PP(\Ext^1(\Lambda_2', \Lambda_1'))$.
Over a point $(\Lambda_3, \Lambda_4) \in \M^{0+}(3m + 2n) \times \M(n + 1)$, $F^{5+}_1$ has fiber $\PP(\Ext^1(\Lambda_3, \Lambda_4))$
and $F^{5-}_1$ has fiber $\PP(\Ext^1(\Lambda_4, \Lambda_3))$.
Over a point $(\Lambda_3', \Lambda_4') \in \M^{0+}(2m + 3n) \times \M(m + 1)$, $F^{5+}_2$ has fiber $\PP(\Ext^1(\Lambda_3', \Lambda_4'))$
and $F^{5-}_2$ has fiber $\PP(\Ext^1(\Lambda_4', \Lambda_3'))$.
Over a point $(\Lambda_5, \Lambda_6) \in \M^{0+}(2m + 2n) \times \M(m + n + 1)$, $F^2$ has fiber $\PP(\Ext^1(\Lambda_5, \Lambda_6))$
and $F^0$ has fiber $\PP(\Ext^1(\Lambda_6, \Lambda_5))$.

\begin{proposition}
\label{flipping_bundles}
The flipping loci are smooth bundles with fibers indicated in Table 1 below.

\begin{table}[ht]{Table 1. Fibers of the flipping loci.}
\begin{center}
\begin{tabular}{| c | c | c | c | c | c | c | c | c | c | c |}
\hline
Bundle & $F^{\infty}_1$ & $F^{11}_1$ & $F^{\infty}_2$ & $F^{11}_2$ & $F^{5+}_1$ & $F^{5-}_1$ & $F^{5+}_2$ & $F^{5-}_2$ & $F^2$ & $F^0$
\\
\hline
Fiber & $\PP^4$ & $\PP^2$ & $\PP^4$ & $\PP^2$ & $\PP^3$ & $\PP^2$ & $\PP^3$ & $\PP^2$ & $\PP^4$ & $\PP^3$
\\
\hline
\end{tabular}
\end{center}
\end{table}
\end{proposition}

\begin{proof}
Choose $\Lambda_1 = (\Gamma_1, \E_1)$ and $\Lambda_2 = (0, \O_L(0, 1))$. From \cite[Corollaire 1.6]{he} we have the exact sequence
\begin{align*}
0 = & \Hom(\Lambda_1, \Lambda_2) \lra \Hom(\E_1, \O_L(0, 1)) \lra \Hom(\Gamma_1, \H^0(\O_L(0, 1))) \simeq \CC^2 \\
\lra & \Ext^1(\Lambda_1, \Lambda_2) \lra \Ext^1(\E_1, \O_L(0, 1)) \lra \Hom(\Gamma_1, \H^1(\O_L(0, 1))) = 0.
\end{align*}
From the short exact sequence
\begin{equation}
\label{E_1}
0 \lra \O(-2, -3) \lra \O \lra \E_1 \lra 0
\end{equation}
we get the long exact sequence
\begin{align*}
0 \lra & \Hom(\E_1, \O_L(0, 1)) \lra \H^0(\O_L(0, 1)) \simeq \CC^2 \lra \H^0(\O_L(2, 4)) \simeq \CC^5 \\
\lra & \Ext^1(\E_1, \O_L(0, 1)) \lra \H^1(\O_L(0, 1)) = 0.
\end{align*}
Combining these exact sequences we obtain $\Ext^1(\Lambda_1, \Lambda_2) \simeq \CC^5$.

From \cite[Corollaire 1.6]{he} we have the exact sequence
\begin{multline*}
0 = \Hom(0, \H^0(\E_1)/\Gamma_1) \lra \Ext^1(\Lambda_2, \Lambda_1) \\
\lra \Ext^1(\O_L(0, 1), \E_1) \lra \Hom(0, \H^1(\E_1)) = 0.
\end{multline*}
From the short exact sequence
\[
0 \lra \O(-1, 1) \lra \O(0, 1) \lra \O_L(0, 1) \lra 0
\]
we get the long exact sequence
\begin{multline*}
0 = \H^0(\E_1(1, -1)) \lra \Ext^1(\O_L(0, 1), \E_1) \lra \H^1(\E_1(0, -1)) \simeq \CC^3 \\
\lra \H^1(\E_1(1, -1)) = 0.
\end{multline*}
From these exact sequences we see that $\Ext^1(\Lambda_2, \Lambda_1) \simeq \CC^3$.

Choose $\Lambda_3 = (\Gamma_3, \E_3)$ and $\Lambda_4 = (0, \O_L)$. From \cite[Corollaire 1.6]{he} we have the exact sequence
\begin{align*}
0 = & \Hom(\Lambda_3, \Lambda_4) \lra \Hom(\E_3, \O_L) \lra \Hom(\Gamma_3, \H^0(\O_L)) \simeq \CC \\
\lra & \Ext^1(\Lambda_3, \Lambda_4) \lra \Ext^1(\E_3, \O_L) \lra \Hom(\Gamma_3, \H^1(\O_L)) = 0.
\end{align*}
From resolution (\ref{E_3}) we get the long exact sequence
\begin{align*}
0 \lra & \Hom(\E_3, \O_L) \lra \H^0(\O_L(1, 2) \oplus \O_L) \simeq \CC^4 \lra \H^0(\O_L(2, 2) \oplus \O_L(1, 3)) \simeq \CC^7 \\
\lra & \Ext^1(\E_3, \O_L) \lra \H^1(\O_L(1, 2) \oplus \O_L) = 0.
\end{align*}
Combining the last two exact sequences we obtain $\Ext^1(\Lambda_3, \Lambda_4) \simeq \CC^4$.

From \cite[Corollaire 1.6]{he} we have the exact sequence
\begin{multline*}
0 = \Hom(0, \H^0(\E_3)/\Gamma_3) \lra \Ext^1(\Lambda_4, \Lambda_3) \\
\lra \Ext^1(\O_L, \E_3) \simeq \Ext^1(\E_3, \O_L(-2, -2))^* \lra \Hom(0, \H^1(\E_3)) = 0.
\end{multline*}
From resolution (\ref{E_3}) we get the long exact sequence
\begin{align*}
0 = & \Hom(\E_3, \O_L(-2, -2)) \lra \H^0(\O_L(-1, 0) \oplus \O_L(-2, -2)) \simeq \CC \lra \H^0(\O_L \oplus \O_L(-1, 1)) \simeq \CC^3 \\
\lra & \Ext^1(\E_3, \O_L(-2, -2)) \lra \H^1(\O_L(-1, 0) \oplus \O_L(-2, -2)) \simeq \CC \lra \H^1(\O_L \oplus \O_L(-1, 1)) = 0.
\end{align*}
From these exact sequences we get $\Ext^1(\Lambda_4, \Lambda_3) \simeq \CC^3$.

Choose $\Lambda_5 = (\Gamma_5, \E_5)$ and $\Lambda_6 = (\Gamma_6, \E_6)$.
From \cite[Corollaire 1.6]{he} we have the exact sequence
\begin{align*}
0 = & \Hom(\Lambda_5, \Lambda_6) \lra \Hom(\E_5, \E_6) \lra \Hom(\Gamma_5, \H^0(\E_6)) \simeq \CC \\
\lra & \Ext^1(\Lambda_5, \Lambda_6) \lra \Ext^1(\E_5, \E_6) \lra \Hom(\Gamma_5, \H^1(\E_6)) = 0.
\end{align*}
From the short exact sequence
\begin{equation}
\label{E_5}
0 \lra \O(-2, -2) \lra \O \lra \E_5 \lra 0
\end{equation}
we get the long exact sequence
\begin{align*}
0 \lra & \Hom(\E_5, \E_6) \lra \H^0(\E_6) \simeq \CC \lra \H^0(\E_6(2, 2)) \simeq \CC^5 \\
\lra & \Ext^1(\E_5, \E_6) \lra \H^1(\E_6) = 0.
\end{align*}
From these exact sequences we get $\Ext^1(\Lambda_5, \Lambda_6) \simeq \CC^5$.

From \cite[Corollaire 1.6]{he} we have the exact sequence
\[
0 = \Hom(0, \H^0(\E_5)/\Gamma_5)
\lra \Ext^1(\Lambda_6, \Lambda_5) \lra \Ext^1(\E_6, \E_5) \lra \Hom(0, \H^1(\E_5)) = 0.
\]
From the short exact sequence
\[
0 \lra \O(-1, -1) \lra \O \lra \E_6 \lra 0
\]
we get the long exact sequence
\begin{align*}
0 = & \Hom(\E_6, \E_5) \lra \H^0(\E_5) \simeq \CC \lra \H^0(\E_5(1, 1)) \simeq \CC^4 \\
\lra & \Ext^1(\E_6, \E_5) \lra \H^1(\E_5) \simeq \CC \lra \H^1(\E_5(1, 1)) = 0.
\end{align*}
From these exact sequences we get $\Ext^1(\Lambda_6, \Lambda_5) \simeq \CC^4$.
\end{proof}

\begin{lemma}
\label{ext^2}
\begin{enumerate}
\item[(i)] For $\Lambda \in F^{11}$ we have $\Ext^2(\Lambda, \Lambda) = 0$.
\item[(ii)] For $\Lambda \in F^{5-}$ we have $\Ext^2(\Lambda, \Lambda) = 0$.
\item[(iii)] For $\Lambda \in F^0$ we have $\Ext^2(\Lambda, \Lambda) = 0$.
\end{enumerate}
\end{lemma}

\begin{proof}
(i) It is enough to consider the case when $\Lambda \in F^{11}_1$, the case when $\Lambda \in F^{11}_2$ being obtained by symmetry.
In view of the exact sequence
\[
0 \lra \Lambda_1 \lra \Lambda \lra \Lambda_2 \lra 0
\]
it suffices to show that $\Ext^2(\Lambda_i, \Lambda_j) = 0$ for $i, j = 1, 2$.
From \cite[Corollaire 1.6]{he} we have the exact sequence
\begin{multline*}
0 = \Hom(\Gamma_1, \H^1(\O_L(0, 1))) \lra \Ext^2(\Lambda_1, \Lambda_2) \\
\lra \Ext^2(\E_1, \O_L(0, 1)) \simeq \Hom(\O_L(0, 1), \E_1 \tensor \omega)^*.
\end{multline*}
The group on the right vanishes because $\O_L$ is stable, by \cite[Proposition 3.2]{genus_two}, $\E_1 \tensor \omega$ is stable,
and $\p(\O_L(0, 1)) > \p(\E_1 \tensor \omega)$. Thus, $\Ext^2(\Lambda_1, \Lambda_2) = 0$.
From the exact sequence
\begin{multline*}
0 = \Hom(0, \H^1(\E_1)) \lra \Ext^2(\Lambda_2, \Lambda_1) \\
\lra \Ext^2(\O_L(0, 1), \E_1) \simeq \Hom(\E_1, \O_L(-2, -1))^* = 0
\end{multline*}
we get the vanishing of $\Ext^2(\Lambda_2, \Lambda_1)$.
From the exact sequence
\begin{multline*}
0 = \Hom(0, \H^1(\O_L(0, 1))) \lra \Ext^2(\Lambda_2, \Lambda_2) \\
\lra \Ext^2(\O_L(0, 1), \O_L(0, 1)) \simeq \Hom(\O_L(0, 1), \O_L(-2, -1))^* = 0
\end{multline*}
we get the vanishing of $\Ext^2(\Lambda_2, \Lambda_2)$.
From \cite[Corollaire 1.6]{he} we have the exact sequence
\begin{align*}
0 = & \Hom(\Gamma_1, \H^0(\E_1)/\Gamma_1) \\
\lra & \Ext^1(\Lambda_1, \Lambda_1) \lra \Ext^1(\E_1, \E_1) \lra \Hom(\Gamma_1, \H^1(\E_1)) \simeq \CC^2 \\
\lra & \Ext^2(\Lambda_1, \Lambda_1) \lra \Ext^2(\E_1, \E_1) \simeq \Hom(\E_1, \E_1 \tensor \omega)^* = 0.
\end{align*}
The space $\Ext^1(\Lambda_1, \Lambda_1)$ is isomorphic to the tangent space of $\M^{0+}(3m + 2n - 1)$ at $\Lambda_1$,
so it is isomorphic to $\CC^{11}$.
From the short exact sequence (\ref{E_1}) we get the long exact sequence
\begin{align*}
0 \lra & \Hom(\E_1, \E_1) \overset{\simeq}{\lra} \H^0(\E_1) \lra \H^0(\E_1(2, 3)) \simeq \CC^{11} \\
\lra & \Ext^1(\E_1, \E_1) \lra \H^1(\E_1) \simeq \CC^2 \lra \H^1(\E_1(2, 3)) = 0.
\end{align*}
Thus, $\Ext^1(\E_1, \E_1) \simeq \CC^{13}$. The vanishing of $\Ext^2(\Lambda_1, \Lambda_1)$ follows from this.

\medskip

\noindent
(ii) From \cite[Corollaire 1.6]{he} we have the exact sequence
\[
0 = \Hom(\Gamma_3, \H^1(\O_L)) \to \Ext^2(\Lambda_3, \Lambda_4) \to \Ext^2(\E_3, \O_L) \simeq \Hom(\O_L, \E_3 \tensor \omega)^* = 0.
\]
Thus, $\Ext^2(\Lambda_3, \Lambda_4) = 0$. From the exact sequence
\[
0 = \Hom(0, \H^1(\E_3)) \lra \Ext^2(\Lambda_4, \Lambda_3) \lra \Ext^2(\O_L, \E_3) \simeq \Hom(\E_3, \O_L \tensor \omega)^* = 0
\]
we get the vanishing of $\Ext^2(\Lambda_4, \Lambda_3)$. From the exact sequence
\[
0 = \Hom(0, \H^1(\O_L)) \to \Ext^2(\Lambda_4, \Lambda_4) \to \Ext^2(\O_L, \O_L) \simeq \Hom(\O_L, \O_L \tensor \omega)^* = 0
\]
we get the vanishing of $\Ext^2(\Lambda_4, \Lambda_4)$. From \cite[Corollaire 1.6]{he} we have the exact sequence
\begin{align*}
0 = & \Hom(\Gamma_3, \H^0(\E_3)/\Gamma_3) \\
\lra & \Ext^1(\Lambda_3, \Lambda_3) \lra \Ext^1(\E_3, \E_3) \lra \Hom(\Gamma_3, \H^1(\E_3)) \simeq \CC \\
\lra & \Ext^2(\Lambda_3, \Lambda_3) \lra \Ext^2(\E_3, \E_3) \simeq \Hom(\E_3, \E_3 \tensor \omega)^* = 0.
\end{align*}
The space $\Ext^1(\Lambda_3, \Lambda_3)$ is isomorphic to the tangent space of $\M^{0+}(3m + 2n)$ at $\Lambda_3$,
so it is isomorphic to $\CC^{12}$. From resolution (\ref{E_3}) we have the long exact sequence
\begin{align*}
0 \to & \Hom(\E_3, \E_3) \simeq \CC \to \H^0(\E_3 \oplus \E_3(1, 2)) \simeq \CC^8 \to \H^0(\E_3(1, 3) \oplus \E_3(2, 2)) \simeq \CC^{19} \\
\to & \Ext^1(\E_3, \E_3) \phantom{\simeq \CC} \ \to \H^1(\E_3 \oplus \E_3(1, 2)) \simeq \CC \phantom{{}^8} \to \H^1(\E_3(1, 3) \oplus \E_3(2, 2)) = 0.
\end{align*}
Thus, $\Ext^1(\E_3, \E_3) \simeq \CC^{13}$. The vanishing of $\Ext^2(\Lambda_3, \Lambda_3)$ follows from this.

\medskip

\noindent
(iii) From \cite[Corollaire 1.6]{he} we have the exact sequence
\[
0 = \Hom(\Gamma_5, \H^1(\E_6)) \lra \Ext^2(\Lambda_5, \Lambda_6) \lra \Ext^2(\E_5, \E_6) \simeq \Hom(\E_6, \E_5 \tensor \omega)^* = 0.
\]
Thus, $\Ext^2(\Lambda_5, \Lambda_6) = 0$. From the exact sequence
\[
0 = \Hom(0, \H^1(\E_5)) \lra \Ext^2(\Lambda_6, \Lambda_5) \lra \Ext^2(\E_6, \E_5) \simeq \Hom(\E_5, \E_6 \tensor \omega)^* = 0
\]
we get the vanishing of $\Ext^2(\Lambda_6, \Lambda_5)$. From the exact sequence
\[
0 = \Hom(0, \H^1(\E_6)) \lra \Ext^2(\Lambda_6, \Lambda_6) \lra \Ext^2(\E_6, \E_6) \simeq \Hom(\E_6, \E_6 \tensor \omega)^* = 0
\]
we get the vanishing of $\Ext^2(\Lambda_6, \Lambda_6)$. From \cite[Corollaire 1.6]{he} we have the exact sequence
\begin{align*}
0 = & \Hom(\Gamma_5, \H^0(\E_5)/\Gamma_5) \\
\lra & \Ext^1(\Lambda_5, \Lambda_5) \lra \Ext^1(\E_5, \E_5) \lra \Hom(\Gamma_5, \H^1(\E_5)) \simeq \CC \\
\lra & \Ext^2(\Lambda_5, \Lambda_5) \lra \Ext^2(\E_5, \E_5) \simeq \Hom(\E_5, \E_5 \tensor \omega)^* = 0.
\end{align*}
The space $\Ext^1(\Lambda_5, \Lambda_5)$ is isomorphic to the tangent space of $\M^{0+}(2m + 2n)$ at $\Lambda_5$,
so it is isomorphic to $\CC^8$.
From the short exact sequence (\ref{E_5}) we get the long exact sequence
\begin{align*}
0 \lra & \Hom(\E_5, \E_5) \overset{\simeq}{\lra} \H^0(\E_5) \phantom{\simeq \CC} \, \lra \H^0(\E_5(2, 2)) \simeq \CC^8 \\
\lra & \Ext^1(\E_5, \E_5) \lra \H^1(\E_5) \simeq \CC \lra \H^1(\E_5(2, 2)) = 0.
\end{align*}
Thus, $\Ext^1(\E_5, \E_5) \simeq \CC^9$. From this we get the vanishing of $\Ext^2(\Lambda_5, \Lambda_5)$.
\end{proof}

\begin{theorem}
\label{wall_crossing}
Let $\MM^{\alpha}$ be the moduli space of $\alpha$-semi-stable pairs on $\PP^1 \times \PP^1$ having Hilbert polynomial $P(m, n) = 3m + 3n + 1$.
We have the following commutative diagrams expressing the variation of $\MM^{\alpha}$ as $\alpha$ crosses the walls:
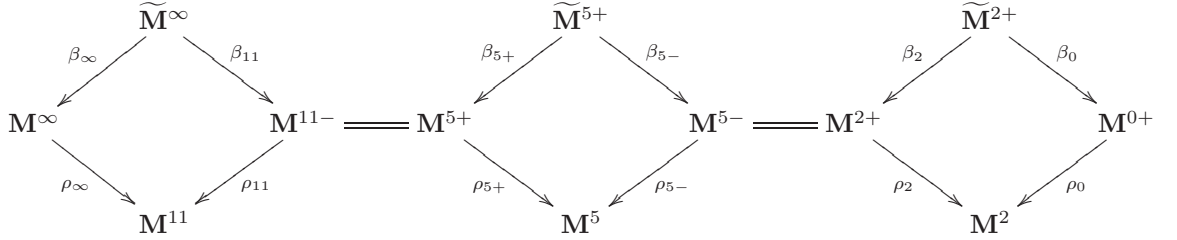
\[
\xymatrix
{
& \widetilde{\MM}^{\infty} \ar[dl]_-{\beta_{\infty}} \ar[dr]^-{\beta_{11}} & & & \widetilde{\MM}^{5+} \ar[dl]_-{\beta_{5+}} \ar[dr]^-{\beta_{5-}} & & & \widetilde{\MM}^{2+} \ar[dl]_-{\beta_2} \ar[dr]^-{\beta_0} \\
\MM^{\infty} \ar[dr]_-{\rho_{\infty}} & & \MM^{11-} \ar[dl]^-{\rho_{11}} \ar@{=}[r] & \MM^{5+} \ar[dr]_-{\rho_{5+}} & & \MM^{5-} \ar[dl]^-{\rho_{5-}} \ar@{=}[r] & \MM^{2+} \ar[dr]_-{\rho_2} & & \MM^{0+} \ar[dl]^-{\rho_0} \\
& \MM^{11} & & & \MM^5 & & & \MM^2
}
\]
Here each $\beta_i$ is the blow-up with center $F^i$.
Moreover, $\beta_{11}$ contracts the component $\widetilde{F}^{\infty}_1$ of the exceptional divisor $\widetilde{F}^{\infty}$ in the direction of $\PP^4$,
where we view $\widetilde{F}^{\infty}_1$ as a $\PP^4 \times \PP^2$-bundle over $\M^{0+}(3m + 2n - 1) \times \M(n + 2)$;
$\beta_{11}$ contracts the component $\widetilde{F}^{\infty}_2$ of $\widetilde{F}^{\infty}$ in the direction of $\PP^4$,
where we view $\widetilde{F}^{\infty}_2$ as a $\PP^4 \times \PP^2$-bundle over $\M^{0+}(2m + 3n - 1) \times \M(m + 2)$;
$\beta_{5-}$ contracts the component $\widetilde{F}^{5+}_1$ of the exceptional divisor $\widetilde{F}^{5+}$ in the direction of $\PP^3$,
where we view $\widetilde{F}^{5+}_1$ as a $\PP^3 \times \PP^2$-bundle over $\M^{0+}(3m + 2n) \times \M(n + 1)$;
$\beta_{5-}$ contracts the component $\widetilde{F}^{5+}_2$ of $\widetilde{F}^{5+}$ in the direction of $\PP^3$,
where we view $\widetilde{F}^{5+}_2$ as a $\PP^3 \times \PP^2$-bundle over $\M^{0+}(2m + 3n) \times \M(m + 1)$;
$\beta_0$ contracts the exceptional divisor $\widetilde{F}^2$ in the direction of $\PP^4$,
where we view $\widetilde{F}^2$ as a $\PP^4 \times \PP^3$-bundle over $\M^{0+}(2m + 2n) \times \M(m + n + 1)$.
The spaces $\MM^{11-}$, $\MM^{5-}$ and $\MM^{0+}$ are smooth.
\end{theorem}

\begin{proof}
The proof of this theorem is analogous to the proof of \cite[Theorem 5.7]{genus_two}.
By Propositions \ref{M^infinity} and \ref{Hilbert_scheme}, $\MM^{\infty}$ is smooth.
We consider the blow-up $\beta_{\infty}$ along the subvariety $F^{\infty}$, which, according to Proposition \ref{flipping_bundles}, is smooth.
We construct $\beta_{11}$, which contracts $\widetilde{F}^{\infty}_1$ into $F^{11}_1$ and $\widetilde{F}^{\infty}_2$ into $F^{11}_2$.
To prove that $\beta_{11}$ is a blow-up with center $F^{11}$ we use the Universal Property of the blow-up, which requires two ingredients:
that $F^{11}$ be smooth and that $\MM^{11-}$ be smooth.
We proved the first property at Proposition \ref{flipping_bundles} and the second property at Lemma \ref{ext^2}(i).
For the other squares in the diagram we proceed analogously.
\end{proof}


\section{Variation of $\M^{\alpha}(3m + 3n - 1)$}
\label{variation_2}

\begin{proposition}
\label{poincare_formula}
We have the following equation of Poincar\'e polynomials:
\[
\Poly(\M(3m + 3n + 1)) = \Poly(\M^{0+}(3m + 3n + 1)) - \xi \Poly(\M^{0+}(3m + 3n - 1)).
\]
\end{proposition}

\noindent
The proof of this proposition is analogous to the proof of \cite[Lemma 5.1]{choi_chung}.
The Poincar\'e polynomial of $\M^{0+}(3m + 3n + 1)$ can be computed by relating it to the Poincar\'e polynomial of $\M^{\infty}(3m + 3n + 1)$
via Theorem \ref{wall_crossing}. Likewise, in order to compute $\Poly(\M^{0+}(3m + 3n -1))$, we will find a relation between $\M^{0+}(3m + 3n -1)$
and $\M^{\infty}(3m + 3n - 1)$.

\begin{proposition}
\label{walls_2}
With respect to the polynomial $P(m, n) = 3m + 3n - 1$ there is exactly one wall at $\alpha = 1$.
\end{proposition}

\begin{proof}
We need to solve the equation
\begin{equation}
\label{alpha_2}
\frac{\alpha + t}{r + s} = \frac{\alpha - 1}{6}
\end{equation}
with integers $0 \le r \le 3$, $0 \le s \le 3$, $t \ge r + s - rs$, $(r, s) \neq (3, 3)$, and $\alpha > 0$ a rational number.
Assume that $(r, s) = (3, 2)$ or $(2, 3)$ and $t \ge -1$.
Equation (\ref{alpha_2}) becomes $\alpha = -5 - 6t$, which has solution $\alpha = 1$ for $t = -1$.
For all other choices of $r$ and $s$ equation (\ref{alpha_2}) has no positive solution in $\alpha$.
\end{proof}

\noindent
For $\alpha \in (1, \infty)$ we write $\M^{\alpha}(3m + 3n - 1) = \M^{\infty}(3m + 3n - 1)$.
For $\alpha \in (0, 1)$ we write $\M^{\alpha}(3m + 3n - 1) = \M^{0+}(3m + 3n - 1)$.
The inclusions of sets of $\alpha$-semi-stable pairs induce the flipping diagram
\[
\xymatrix
{
\M^{\infty}(3m + 3n - 1) \ar[dr]_-{\rho_{\infty}} & & \M^{0+}(3m + 3n - 1) \ar[dl]^-{\rho_0} \\
& \M^1(3m + 3n - 1)
}
\]
in which both maps are birational.

\begin{proposition}
\label{M^infinity_2}
The variety $\M^{\infty}(3m + 3n - 1)$ is isomorphic to the flag Hilbert scheme of zero-dimensional subschemes of length $2$
contained in curves of bidegree $(3, 3)$ in $\PP^1 \times \PP^1$.
\end{proposition}

\noindent
In particular, $\M^{\infty}(3m + 3n - 1)$ is a bundle with base $\Hilb_{\PP^1 \times \PP^1}(2)$ and fiber $\PP^{13}$.

\begin{remark}
\label{flipping_base_2}
From the proof of Proposition \ref{walls_2} we see that the strictly $\alpha$-semi-stable locus in $\M^1(3m + 3n - 1)$
has two connected components:
\[
\M^{0+}(3m + 2n -1) \times \M(n) \qquad \text{and} \qquad \M^{0+}(2m + 3n - 1) \times \M(m).
\]
A point in $\M(n)$ is of the form $\O_L(0, -1)$, where $L \subset \PP^1 \times \PP^1$ is a line of bidegree $(1, 0)$.
\end{remark}

\noindent
Consider the flipping loci
\begin{align*}
F^{\infty}_1 & = \rho_{\infty}^{-1}(\M^{0+}(3m + 2n - 1) \times \M(n)) \subset \M^{\infty}(3m + 3n - 1), \\
F^{\infty}_2 & = \rho_{\infty}^{-1}(\M^{0+}(2m + 3n - 1) \times \M(m)) \subset \M^{\infty}(3m + 3n - 1), \\
F^{\infty} & = F^{\infty}_1 \cup F^{\infty}_2, \\
F^0_1 & = \rho_0^{-1}(\M^{0+}(3m + 2n - 1) \times \M(n)) \subset \M^{0+}(3m + 3n - 1), \\
F^0_2 & = \rho_0^{-1}(\M^{0+}(2m + 3n - 1) \times \M(m)) \subset \M^{0+}(3m + 3n - 1), \\
F^0 & = F^0_1 \cup F^0_2.
\end{align*}
Over a point $(\Lambda_1, \Lambda_7) \in \M^{0+}(3m + 2n - 1) \times \M(n)$, $F^{\infty}_1$ has fiber $\PP(\Ext^1(\Lambda_1, \Lambda_7))$
and $F^0_1$ has fiber $\PP(\Ext^1(\Lambda_7, \Lambda_1))$.
Over a point $(\Lambda_1', \Lambda_7') \in \M^{0+}(2m + 3n - 1) \times \M(m)$, $F^{\infty}_2$ has fiber $\PP(\Ext^1(\Lambda_1', \Lambda_7'))$
and $F^0_2$ has fiber $\PP(\Ext^1(\Lambda_7', \Lambda_1'))$.

\begin{proposition}
\label{flipping_bundles_2}
The flipping loci $F^{\infty}_1$, $F^{\infty}_2$, $F^0_1$, $F^0_2$ are smooth bundles with fiber $\PP^2$.
\end{proposition}

\begin{proof}
Choose $\Lambda_1 = (\Gamma_1, \E_1)$ and $\Lambda_7 = (0, \O_L(0, -1))$.
From \cite[Corollaire 1.6]{he} we have the exact sequence
\begin{multline*}
0 = \Hom(\Gamma_1, \H^0(\O_L(0, -1))) \lra \Ext^1(\Lambda_1, \Lambda_7) \\
\lra \Ext^1(\E_1, \O_L(0, -1)) \lra \Hom(\Gamma_1, \H^1(\O_L(0, -1))) = 0.
\end{multline*}
From resolution (\ref{E_1}) we get the long exact sequence
\begin{multline*}
0 = \H^0(\O_L(0, -1)) \lra \H^0(\O_L(2, 2)) \simeq \CC^3 \lra \Ext^1(\E_1, \O_L(0, -1)) \\
\lra \H^1(\O_L(0, -1)) = 0.
\end{multline*}
Thus, $\Ext^1(\Lambda_1, \Lambda_7) \simeq \CC^3$. From \cite[Corollaire 1.6]{he} we have the exact sequence
\begin{multline*}
0 = \Hom(0, \H^0(\E_1)/\Gamma_1) \lra \Ext^1(\Lambda_7, \Lambda_1) \\
\lra \Ext^1(\O_L(0, -1), \E_1) \lra \Hom(0, \H^1(\E_1)) = 0.
\end{multline*}
From the short exact sequence
\[
0 \lra \O(-1, -1) \lra \O(0, -1) \lra \O_L(0, -1) \lra 0
\]
we get the long exact sequence
\begin{align*}
0 = & \Hom(\O_L(0, -1), \E_1) \lra \H^0(\E_1(0, 1)) \simeq \CC^2 \lra \H^0(\E_1(1, 1)) \simeq \CC^4 \\
\lra & \Ext^1(\O_L(0, -1), \E_1) \lra \H^1(\E_1(0, 1)) \simeq \CC \phantom{{}^2} \lra \H^1(\E_1(1, 1)) = 0.
\end{align*}
From these exact sequences we obtain $\Ext^1(\Lambda_7, \Lambda_1) \simeq \CC^3$.
\end{proof}

\begin{lemma}
\label{ext^2_2}
For $\Lambda \in F^0$ we have $\Ext^2(\Lambda, \Lambda) = 0$.
\end{lemma}

\begin{proof}
It is enough to consider the case when $\Lambda \in F^0_1$, the case when $\Lambda \in F^2_0$ being obtained by symmetry.
In view of the exact sequence
\[
0 \lra \Lambda_1 \lra \Lambda \lra \Lambda_7 \lra 0
\]
it is enough to show that $\Ext^2(\Lambda_i, \Lambda_j) = 0$ for $i, j = 1, 7$.
To obtain these vanishings we proceed exactly as in the proof of Lemma \ref{ext^2}(i) with $\Lambda_2$ replaced by $\Lambda_7$.
\end{proof}

\begin{theorem}
\label{wall_crossing_2}
The space $\M^{0+}(3m + 3n - 1)$ is smooth. We have the following commutative diagram expressing the variation of $\M^{\alpha}(3m + 3n - 1)$
as $\alpha$ crosses the wall:
\[
\xymatrix
{
& \widetilde{\M}(3m + 3n - 1) \ar[dl]_-{\beta_{\infty}} \ar[dr]^-{\beta_0} \\
\M^{\infty}(3m + 3n - 1) \ar[dr]_-{\rho_{\infty}} & & \M^{0+}(3m + 3n - 1) \ar[dl]^-{\rho_0} \\
& \M^1(3m + 3n - 1)
}
\]
Here $\beta_{\infty}$ is the blow-up with center $F^{\infty}$ and $\beta_0$ is the blow-up with center $F^0$.
The exceptional divisor $\widetilde{F}^{\infty}$ has two connected components: $\widetilde{F}^{\infty}_1$ and $\widetilde{F}^{\infty}_2$.
We view $\widetilde{F}^{\infty}_1$ as a $\PP^2 \times \PP^2$-bundle over $\M^{0+}(3m + 2n - 1) \times \M(n)$
and $\widetilde{F}^{\infty}_2$ as a $\PP^2 \times \PP^2$-bundle over $\M^{0+}(2m + 3n - 1) \times \M(m)$.
Thus, $\beta_0$ contracts $\widetilde{F}^{\infty}_1$ in the direction of the first $\PP^2$ into $F^0_1$
and it contracts $\widetilde{F}^{\infty}_2$ in the direction of the first $\PP^2$ into $F^0_2$.
\end{theorem}

\begin{proof}
The proof of this theorem is analogous to the proof of Theorem \ref{wall_crossing} and makes use of Proposition \ref{flipping_bundles_2}
and Lemma \ref{ext^2_2} in a similar fashion.
\end{proof}

\noindent
{\bf Proof of Theorem \ref{poincare_polynomial}.} By Theorem \ref{wall_crossing}, we have
\begin{align*}
\Poly(\MM^{0+}) = \Poly(\MM^{\infty}) & + (\Poly(\PP^2) - \Poly(\PP^4)) \Poly(\M^{0+}(3m + 2n - 1) \times \M(n + 2)) \\
& + (\Poly(\PP^2) - \Poly(\PP^4)) \Poly(\M^{0+}(2m + 3n - 1) \times \M(m + 2)) \\
& + (\Poly(\PP^2) - \Poly(\PP^3)) \Poly(\M^{0+}(3m + 2n) \times \M(n + 1)) \\
& + (\Poly(\PP^2) - \Poly(\PP^3)) \Poly(\M^{0+}(2m + 3n) \times \M(m + 1)) \\
& + (\Poly(\PP^3) - \Poly(\PP^4)) \Poly(\M^{0+}(2m + 2n) \times \M(m + n + 1)).
\end{align*}
By Proposition \ref{M^infinity}, Proposition \ref{Hilbert_scheme} and Remark \ref{flipping_base}, we have further
\begin{align*}
\Poly(\MM^{0+}) = \Poly(\PP^{11}) \Poly(\Hilb_{\PP^1 \times \PP^1}(4)) & + 2 (\Poly(\PP^2) - \Poly(\PP^4)) \Poly(\PP^{11}) \Poly(\PP^1) \\
& + 2 (\Poly(\PP^2) - \Poly(\PP^3)) \Poly(\PP^{10}) \Poly(\PP^1 \times \PP^1) \Poly(\PP^1) \\
& + \phantom{2} (\Poly(\PP^3) - \Poly(\PP^4)) \Poly(\PP^8) \Poly(\PP^3).
\end{align*}
By Theorem \ref{wall_crossing_2}, we have
\begin{align*}
\Poly(\M^{0+}(3m + 3n - 1)) = & \Poly(\M^{\infty}(3m + 3n - 1) \\
& + (\Poly(\PP^2) - \Poly(\PP^2)) \Poly(\M^{0+}(3m + 2n - 1) \times \M(n)) \\
& + (\Poly(\PP^2) - \Poly(\PP^2)) \Poly(\M^{0+}(2m + 3n - 1) \times \M(m)) \\
= & \Poly(\M^{\infty}(3m + 3n - 1).
\end{align*}
By Proposition \ref{M^infinity_2}, we have further
\[
\Poly(\M^{0+}(3m + 3n - 1)) = \Poly(\PP^{13}) \Poly(\Hilb_{\PP^1 \times \PP^1}(2)).
\]
According to \cite[Theorem 0.1]{goettsche},
\begin{align*}
\Poly(\Hilb_{\PP^1 \times \PP^1}(2)) & = \xi^4 + 3\xi^3 + 6\xi^2 + 3\xi + 1, \\
\Poly(\Hilb_{\PP^1 \times \PP^1}(4)) & = \xi^8 + 3\xi^7 + 10\xi^6 + 22\xi^5 + 33\xi^4 + 22\xi^3 + 10\xi^2 + 3\xi + 1.
\end{align*}
From Proposition \ref{poincare_formula} we finally obtain
\begin{align*}
\Poly(\MM) = & \frac{\xi^{12} - 1}{\xi - 1}(\xi^8 + 3\xi^7 + 10\xi^6 + 22\xi^5 + 33\xi^4 + 22\xi^3 + 10\xi^2 + 3\xi + 1) \\
& - 2(\xi^3 + \xi^4) \frac{\xi^{12} - 1}{\xi - 1} (\xi + 1) - 2\xi^3 \frac{\xi^{11} - 1}{\xi - 1} (\xi + 1)^3 - \xi^4 \frac{\xi^9 - 1}{\xi - 1} \frac{\xi^4 - 1}{\xi - 1} \\
& - \xi \frac{\xi^{14} - 1}{\xi - 1} (\xi^4 + 3\xi^3 + 6\xi^2 + 3\xi + 1).
\end{align*}

\end{document}